\documentclass{amsart}

\usepackage{graphics}
\usepackage{hyperref}

\newtheorem{theorem}{Theorem}[section]
\newtheorem{lemma}[theorem]{Lemma}

\theoremstyle{definition}
\newtheorem{definition}[theorem]{Definition}

\newtheorem{corollary}[theorem]{Corollary}
\newtheorem{proposition}[theorem] {Proposition}
\theoremstyle{remark}
\newtheorem*{remark}{Remark}

\numberwithin{equation}{section}

\begin{document}

\title[Combinatorial realization of the Thom-Smale complex]{Combinatorial realization of the Thom-Smale complex via discrete Morse theory}

\author{\'Etienne Gallais}
\address{Laboratoire de Math\'ematiques et Applications des Mathématiques (LMAM), Universit\'e de Bretagne Sud, Campus de Tohannic - BP 573, 56017 Vannes,  FRANCE}
\curraddr{Laboratoire de Math\'ematiques Jean Leray (LMJL)
UFR Sciences et Techniques, 2 rue de la Houssini\`eres - BP 92208, 44 322 Nantes Cedex 3, FRANCE}
\email{etienne.gallais@math.univ-nantes.fr}

%\author{Author Two}
%\address{Mathematical Research Section, School of Mathematical Sciences,
%Australian National University, Canberra ACT 2601, Australia}

%    General info
%\subjclass[2000]{Primary 54C40, 14E20; Secondary 46E25, 20C20}

%\date{January 1, 2001 and, in revised form, June 22, 2001.}


%\keywords{Differential geometry, algebraic geometry}

\begin{abstract}
In the case of smooth manifolds, we use Forman's discrete Morse theory to realize combinatorially any Thom-Smale complex coming from a smooth Morse function by a couple triangulation-discrete Morse function.
As an application, we prove that any Euler structure on a smooth oriented closed 3-manifold has a particular realization by a complete matching on the Hasse diagram of a triangulation of the manifold.
\end{abstract}

\maketitle

\section{Introduction}
R. Forman defines a combinatorial analog of smooth Morse theory in \cite{Forman1}, \cite{Forman2}, \cite{Forman3} for simplicial complexes and more generally for CW-complexes.
Discrete Morse theory has many applications (computer graphics \cite{lewiner-lopes-tavares-tvcg04}, graph theory \cite{farley-2005-5}).
An important problem is the research of optimal discrete Morse functions in the sense that they have the minimal number of critical cells (\cite{joswig-2004}, \cite{lewiner-lopes-tavares-cgta03},\cite{salvetti-2007} for the minimality of hyperplane arrangements).

Thanks to a combinatorial Morse vector field $V$, Forman constructs a combinatorial Thom-Smale complex $(C^V,\partial^V)$ whose homology is the simplicial homology of the simplicial complex.
The differential is defined by counting algebraically $V$-paths between critical cells.
Nevertheless, the proof of $\partial^V \circ \partial^V=0$ is an indirect proof (see \cite{Forman2}).
We give two proofs of $\partial^V \circ \partial^V=0$, one which focusses on the geometry and another one which focusses on the algebraic point of view (compare with \cite{chari} and \cite{skoldberg}).
In fact, the algebraic proof gives also the property that the combinatorial Thom-Smale complex is a chain complex homotopy equivalent to the simplicial chain complex.
After that, we investigate one step forward the relation between the smooth Morse theory and the discrete Morse theory.
We prove that any Thom-Smale complex has a combinatorial realization.
We use this to prove that any {${\rm Spin}\sp c$}-structures on a closed oriented 3-manifold can be realized by a complete matching on a triangulation of this manifold.

This article is organised as follows.
In section \ref{section:discrete-morse-theory}, we recall the discrete Morse theory from the viewpoint of combinatorial Morse vector field.
Section \ref{section:proof of d2 equal 0} is devoted to the proofs of $\partial^V \circ \partial^V=0$ and that the Thom-Smale complex is a chain complex homotopy equivalent to the simplicial chain complex.
In section \ref{section:link-between-discrete-and-smooth}, we prove that any combinatorial Thom-Smale complex is realizable as a combinatorial Thom-Smale complex.
In section \ref{section:complete-matchings-and-euler-structures}, we obtain as a corollary the existence of triangulations with complete matchings on their Hasse diagram and prove that any {${\rm Spin}\sp c$}-structures on a closed oriented 3-manifold can be realized by such complete matchings.

\section{Discrete Morse theory}\label{section:discrete-morse-theory}
\subsection{Combinatorial Morse vector field}
First of all, instead of considering discrete Morse functions on a simplicial complex, we will only consider combinatorial Morse vector fields.
In fact, working with discrete Morse functions or combinatorial vector fields is exactly the same \cite[Theorem 9.3]{Forman2}.

In the following, $X$ is a finite simplicial complex and $K$ is the set of cells of $X$.
A cell $\sigma \in K$ of dimension $k$ is denoted $\sigma^{(k)}$.
Let $<$ be the partial order on $K$ given by $\sigma < \tau$ iff $\sigma \subset \overline{\tau}$.
Given a simplicial complex, one associates its Hasse diagram: the set of vertices is the set of cells $K$, an edge joins two cells $\sigma$ and $\tau$ if $\sigma<\tau$ and $dim(\sigma)+1 =dim(\tau)$.

\begin{definition}
A combinatorial vector field $V$ on $X$ is an oriented matching on the associated Hasse diagram of $X$ that is a set of edges $\mathcal M$ such that
\begin{enumerate}
 \item any two distincts edges of $\mathcal M$ do not share any common vertex,
 \item every edge belonging to $\mathcal M$ is oriented toward the top dimensional cell.
\end{enumerate}
A cell which does not belong to any edge of the matching is said to be critical.
\end{definition}

\begin{remark}
The original definition of a combinatorial vector field is the following one: given a matching on the Hasse diagram define
$$
\begin{tabular}{rcll}
$V:$ & $K$ & $\rightarrow$ & $K\cup\{0\}$ \\
& $\sigma$ & $\mapsto$ & $V(\sigma)=  \begin{cases}
\tau & \text{ iff }(\sigma,\tau)\text{ is an edge of the matching and }\sigma<\tau, \\
0 & \text{ otherwise.}\\
\end{cases}$ \\
\end{tabular}
$$
We will use both these points of view in the following.
\end{remark}

A $V$-path of dimension $k$ is a sequence of cells $\gamma : \sigma_0 ,\sigma _1 , \ldots , \sigma_r$ of dimension $k$ such that
\begin{enumerate}
  \item $\sigma_i \neq \sigma_{i+1}$ for all $i\in \{0,\ldots,r-1\}$,
  \item for every $i\in \{0,\ldots,r-1\}$, $\sigma_{i+1} < V(\sigma_i)$.
\end{enumerate}
A $V$-path $\gamma$ is said to be closed if $\sigma_0 = \sigma_r$, and non-stationary if $r>0$.

\begin{definition}
 A combinatorial vector field $V$ which has no non-stationary closed path is called a combinatorial Morse vector\index{Combinatorial!Morse vector field}.
In this case, the corresponding matching is called a Morse matching.
\end{definition}

The terminology \emph{Morse matching} first appeared in \cite{chari}.

\begin{remark}\label{rem:removing-egde}
 Let $V$ be a combinatorial (resp. combinatorial Morse) vector field.
It we remove an edge from the underlying matching, it remains a combinatorial (resp. combinatorial Morse) vector field (there are two extra critical cells).
\end{remark}

\subsection{The combinatorial Thom-Smale complex}
\subsubsection{Definition of the combinatorial Thom-Smale complex}
The following data are necessary to define the combinatorial Thom-Smale complex (see \cite{Forman2}).
First, let $X$ be a finite simplicial complex, $K$ its set of cells and $V$ a combinatorial Morse vector field.
Suppose that every cell $\sigma\in K$ is oriented.

Let $\gamma : \sigma_0 ,\sigma _1 , \ldots , \sigma_r$ be a $V$-path.
Then the multiplicity of $\gamma$ is given by the formula
$$
m(\gamma)=\prod_{ i=0 }^{r-1} -< \partial V(\sigma_i),\sigma_i > < \partial V(\sigma_i),\sigma_{i+1}> \qquad \in \{\pm1\}
$$
where for every cell $\sigma, \tau$, $<\sigma,\tau>\in \{-1,0,1\}$ is the incidence number between the cells $\sigma$ and $\tau$ (see \cite{LW}) and $\partial$ is the boundary map when we consider $X$ as a CW-complex.
In fact, one can think of the multiplicity as checking if the orientation of the first cell $\sigma_0$ moved along $\gamma$ coincides or not with the orientation of the last cell $\sigma_r$.

Let $\Gamma(\sigma,\sigma')$ be the set of $V$-paths starting at $\sigma$ and ending at $\sigma'$ and $Crit_{k}(V)$ be the set of critical cells of dimension $k$.

\begin{definition}
 The combinatorial Thom-Smale complex\index{Combinatorial!Thom-Smale complex} associated with $(X,V)$ is $(C^{V}_{*},\partial^V )$ where:
\begin{enumerate}
\item $C^{V}_{k}=\bigoplus_{\sigma \in Crit_{k}(V)} \mathbb{Z}. \sigma$,
\item if $\tau \in Crit_{k+1}(V)$ then
$$\partial^V \tau = \displaystyle{\sum_{\sigma \in Crit_{k}(V)} n(\tau,\sigma). \sigma}$$
where
$$
n(\tau,\sigma) = \sum_{\widetilde{\sigma} <\tau}  <\partial \tau, \widetilde{\sigma}> \sum_{\gamma \in \Gamma(\widetilde{\sigma},\sigma)} m(\gamma)
$$
\end{enumerate}
\end{definition}

Thus, this complex is exactly in the same spirit as the Thom-Smale complex for smooth Morse functions (see section \ref{section:link-between-discrete-and-smooth}): it is generated by critical cells and the differential is given by counting algebraically $V$-paths.

\begin{theorem}[Forman \cite{Forman2}]\label{theo:d2=0}
 $\partial^V \circ \partial^V = 0$.
\end{theorem}

\begin{theorem}[Forman \cite{Forman2}]\label{theo:hom=hom-simplicial}
$(C^{V}_{*},\partial^V )$ is homotopy equivalent to the simplicial chain complex.
In particular, its homology is equal to the simplicial homology.
\end{theorem}

We will give a direct proof of both of these theorems.
The proof of Theorem \ref{theo:d2=0} is done by looking at $V$-paths and understanding their contribution to $\partial^V \circ \partial^V$.
Then, we prove Theorem \ref{theo:hom=hom-simplicial} (which gives another proof of Theorem \ref{theo:d2=0}) using Gaussian elimination (this idea first appears in \cite{chari}, see also \cite{skoldberg}).

\subsubsection{Proof of Theorem \ref{theo:d2=0}}\label{section:proof-of-d2-equal-0}
Let $X$ be a simplicial complex, $K$ be the set of its cells and $K_n$ the set of cells of dimension $n$.
The proof is by induction on the number on edges belonging to the Morse matching.

\paragraph{\textbf{Initialization:} matching with no edge.}
In this case, every cell is critical and the combinatorial Thom-Smale complex coincides with the well-known simplicial chain complex.
Therefore, $\partial^V \circ \partial^V =0$.

\paragraph{\textbf{Heredity:} suppose the property is true for every matching with at most $k$ edges defining a combinatorial Morse vector field.}
Let $V$ be a combinatorial Morse vector field with corresponding matching consisting of $k+1$ edges.
In particular, there is no non-stationary closed $V$-path.
Let $(\sigma,\tau)$ be an edge of this matching with $\sigma<\tau$ and let $\overline V$ be the combinatorial Morse vector field corresponding to the original matching with the edge $ (\sigma,\tau)$ removed.
By induction hypothesis $\partial^{\overline V}\circ  \partial^{\overline V} = 0$.
In particular, for every $n\in \mathbb{N}$, every $\tau \in K_{n+1}$ and every $\nu\in K_{n-1}$ when there is a cell $\sigma_1\in K_n$ such that there is a $\overline{V}$-path from an hyperface of $\tau$ to $\sigma_1$ and another $\overline{V}$-path from an hyperface of $\sigma_1$ to $\nu$ there exists another cell $\sigma_2 \in K_n$ with the same property so that their contribution to $\partial^{\overline V}\circ  \partial^{\overline V}$ are opposite.

First, we will prove that $\partial^V \circ \partial^V= 0$ when the chain complex is with coefficients in $\mathbb{Z}/2\mathbb{Z}$ and after we will take care of signs.

Suppose that the distinguished edge of the matching $(\sigma,\tau)$ is such that $dim(\sigma)+1=dim(\tau)=n+1$.
Therefore, $C^{V}_{i} = C^{\overline{V}}_{i}$ for $i\neq n , n+1$ and $\partial^V_{\vert C^{V}_{i}} = \partial^{\overline{V}}_{\vert  C^{V}_{i}}$ for $i\notin \{n,n+1,n+2\}$.
So we have $\partial^V \circ \partial^V (\mu) = 0$ for all $\mu \in K - (K_{n}\cup K_{n+1}\cup K_{n+2})$.

Remark that it is also true for every $\sigma'\in Crit_{n}(V)$ that $\partial^V \circ \partial^V (\sigma')=0$ (since with respect to $\overline V$ it is true and $\sigma'\neq \sigma$).

There are two cases left.
\begin{itemize}
 \item[\textbf{Case 1.}] Let $\tau'\in Crit_{n+1}(V)$.
To see that $\partial^V\circ \partial^V(\tau')=0$ we must consider two cases.
First case is when the two $\overline{V}$-paths which annihilates don't go through $\sigma$.
Then, nothing is changed and contributions to $\partial^V\circ \partial^V(\tau')$ cancel by pair.
The second case is when at least one the $\overline{V}$-path which cancel by pair for $\partial^{\overline{V}}$ go through $\sigma$.
The $\overline{V}$-paths which go from $\tau'$ to $\nu$ are of two types: those who go via $\sigma$ and the others.
Let $\tau' \rightarrow \sigma_2 \rightarrow \nu$ be a juxtaposition of two $\overline{V}$-paths which cancel with the juxtaposition of $\overline{V}$-path $\tau' \rightarrow \sigma \rightarrow \nu$.
Since $\partial^{\overline V}\circ \partial^{\overline{V}}(\tau)=0$, there must be a critical cell $\sigma_1$ such that the juxtaposition of $\overline{V}$-paths $\tau\rightarrow \sigma \rightarrow \nu$ and $\tau\rightarrow \sigma_1 \rightarrow \nu$ cancels.
Therefore, when considering $\partial^V$, three juxtapositions of $\overline{V}$-paths disappear and one is created: $\tau' \rightarrow (\sigma\rightarrow \tau) \rightarrow \sigma_1 \rightarrow \nu$.
It cancels with $\tau' \rightarrow \sigma_2 \rightarrow \nu$.

It may happens that two juxtapositions of $\overline{V}$-paths go through $\sigma$ but this case works exactly in the same way.

\item[\textbf{Case 2.}] This case is similar to the previous case.
Let $\varsigma$ be a cell in $K_{n+2}$.
There are two cases to see that $\partial^V\circ \partial^V (\varsigma)=0$.
The first case is when the two $\overline{V}$-paths whose contributions are opposite don't go through $\tau$.
Then, nothing is changed and contributions to $\partial^V\circ \partial^V(\tau')$ cancel by pair.
The second case is when the $\overline{V}$-path which disappears is replaced by exactly a new one which goes through the edge $(\sigma,\tau)$.
The result follows similarly.
\end{itemize}

Note that to deal with this two cases we used the fact that there is no non-stationary closed $V$-path (and so $\overline{V}$-path).
Now, let's deal with the signs.
We will only consider the case 1. above, other cases work similarly.
Denote $n(\alpha\rightarrow \beta)$ (resp. $\overline{n}(\alpha\rightarrow \beta)$) the sign of the contribution in the differential $\partial^V$ (resp. $\partial^{\overline{V}}$) of a path going from $\alpha$ to $\beta$ where both cells are critical of consecutive dimension.
While considering $\overline{V}$, we have by induction hypothesis
\begin{equation}\label{eq:diff-1}
 \overline{n}(\tau'\rightarrow \sigma_2).\overline{n}(\sigma_2\rightarrow \nu) = - \overline{n}(\tau'\rightarrow \sigma).\overline{n}(\sigma\rightarrow\nu)
\end{equation}
 and
\begin{equation}\label{eq:diff-2}
 \overline{n}(\tau\rightarrow \sigma_1).\overline{n}(\sigma_1\rightarrow \nu) = - \overline{n}(\tau\rightarrow \sigma).\overline{n}(\sigma\rightarrow\nu)
\end{equation}

Since the juxtaposition of the $\overline{V}$-paths $\tau'\rightarrow \sigma_2 \rightarrow \nu$ don't go through $\sigma$ we have that
\begin{equation}\label{eq:diff-3}
 \overline{n}(\tau'\rightarrow \sigma_2).\overline{n}(\sigma_2\rightarrow \nu) = n(\tau'\rightarrow \sigma_2).n(\sigma_2\rightarrow\nu)
\end{equation}

By definition of the multiplicity of paths we have
\begin{equation}\label{eq:diff-4}
 n(\tau'\rightarrow \sigma_1) = \overline{n}(\tau'\rightarrow \sigma) . (- <\partial \tau ,\sigma> ) . \overline{n}(\tau \rightarrow \sigma_1 )
\end{equation}

Combining equations \eqref{eq:diff-1}-\eqref{eq:diff-4} we obtain the following equalities
$$
\begin{array}{rcl}
n(\tau'\rightarrow \sigma_1).n(\sigma_1\rightarrow \nu) & = & \overline{n}(\tau'\rightarrow \sigma) . (- <\partial \tau ,\sigma> ) . \overline{n}(\tau \rightarrow \sigma_1 ) . n(\sigma_1 \rightarrow \nu)\eqref{eq:diff-4} \\

& = &\overline{n}(\tau'\rightarrow \sigma) . (- <\partial \tau ,\sigma> ) . \overline{n}(\tau \rightarrow \sigma_1 ) . \overline{n}(\sigma_1 \rightarrow \nu) \\

& = & \overline{n}(\tau'\rightarrow \sigma) . <\partial \tau ,\sigma>  . \overline{n}(\tau \rightarrow \sigma ) . \overline{n}(\sigma \rightarrow \nu) \quad \eqref{eq:diff-2}\\

& = & (<\partial \tau ,\sigma>  . \overline{n}(\tau \rightarrow \sigma ) ) .\overline{n}(\tau'\rightarrow \sigma) . \overline{n}(\sigma \rightarrow \nu) \\

& = & \overline{n}(\tau'\rightarrow \sigma) . \overline{n}(\sigma \rightarrow \nu) \quad\text{by definition}\\

& = & -\overline{n}(\tau'\rightarrow \sigma_2) . \overline{n}(\sigma_2 \rightarrow \nu) \quad \eqref{eq:diff-1}\\

& = & -n(\tau'\rightarrow \sigma_2) . n(\sigma_2 \rightarrow \nu) \quad \eqref{eq:diff-3}\\
\end{array}
$$
which concludes the proof of the theorem.

\subsubsection{Proof of Theorem \ref{theo:hom=hom-simplicial}}\label{section:proof of d2 equal 0}
The main ingredient of the proof is thinking about combinatorial Morse vector field as an instruction to remove acyclic complexes from the original simplicial chain complex, as done by Chari \cite[Proposition 3.3]{chari} or Sk\"oldberg \cite{skoldberg}.
Given a matching between two cells $\sigma<\tau$, we would like to remove the following short complex (which is acyclic)
$$
0 \rightarrow \mathbb{Z}.\tau \xrightarrow[]{<\partial \tau,\sigma>} \mathbb{Z}.\sigma \rightarrow 0
$$
where $\partial$ is the boundary operator of the simplicial chain complex.
To do this, we use Gaussian elimination (see e.g. \cite{barnatan-2006}):

\begin{lemma}[Gaussian elimination]\label{lemma:gaussian-elimination}%[Gaussian elimination]
Let $\mathcal C = (C_* ,\partial)$ be a chain complex over $\mathbb{Z}$ freely generated.
Let $b_1 \in C_{i}$ (resp. $b_2\in C_{i-1}$) be such that $C_i = \mathbb{Z}.b_1 \oplus D$ (resp. $C_{i-1} = \mathbb{Z}.b_2 \oplus E$).
If $\phi:\mathbb{Z}.b_1 \rightarrow \mathbb{Z}.b_2$ is an isomorphism of $\mathbb{Z}$-modules, then the four term complex segment of $\mathcal C$
\begin{equation}\label{eq:complex-segment-1}
 \ldots \rightarrow
\begin{bmatrix} C_{i+1} \end{bmatrix}
\xrightarrow[]{ \begin{pmatrix} \alpha \\ \beta \end{pmatrix} }
{ \begin{bmatrix} b_1 \\ D \end{bmatrix} }
\xrightarrow[]{ \begin{pmatrix} \phi & \delta \\ \gamma & \varepsilon \end{pmatrix}  }
{ \begin{bmatrix} b_2 \\ E \end{bmatrix} }
\xrightarrow[]{ \begin{pmatrix}\mu & \nu \end{pmatrix}  }
{ \begin{bmatrix} C_{i-2} \end{bmatrix} }
\rightarrow \ldots
\end{equation}
is isomorphic to the following chain complex segment
\begin{equation}\label{eq:complex-segment-2}
 \ldots \rightarrow
{ \begin{bmatrix} C_{i+1} \end{bmatrix} }
\xrightarrow[]{ \begin{pmatrix} 0 \\ \beta \end{pmatrix} }
{ \begin{bmatrix} b_1 \\ D \end{bmatrix} }
\xrightarrow[]{ \begin{pmatrix} \phi & 0 \\ 0 & \varepsilon - \gamma\phi^{-1}\delta \end{pmatrix}  }
{ \begin{bmatrix} b_2 \\ E \end{bmatrix} }
\xrightarrow[]{ \begin{pmatrix} 0 & \nu \end{pmatrix}  }
{ \begin{bmatrix} C_{i-2} \end{bmatrix} }
\rightarrow \ldots
\end{equation}
Both these complexes are homotopy equivalent to the complex segment
\begin{equation}\label{eq:complex-segment-3}
 \ldots \rightarrow
{ \begin{bmatrix} C_{i+1} \end{bmatrix} }
\xrightarrow[]{ \begin{pmatrix}  \beta \end{pmatrix} }
{ \begin{bmatrix}  D \end{bmatrix} }
\xrightarrow[]{ \begin{pmatrix}  \varepsilon - \gamma\phi^{-1}\delta \end{pmatrix} }
{ \begin{bmatrix}  E \end{bmatrix} }
\xrightarrow[]{ \begin{pmatrix}  \nu \end{pmatrix}  }
{ \begin{bmatrix} C_{i-2} \end{bmatrix} }
\rightarrow \ldots
\end{equation}
Here we used matrix notation for the differential $\partial$.
\end{lemma}

\begin{proof}
Since $\partial^2 = 0$ in $\mathcal C$, we obtain $\phi\alpha + \delta\beta=0$ and $\mu \phi +\nu\gamma=0$.
By doing the following change of basis $A = \begin{pmatrix} 1 & \phi^{-1}\delta \\ 0 & 1 \end{pmatrix}$ on $\begin{bmatrix} b_1 \\ D \end{bmatrix}$ and $B = \begin{pmatrix} 1 & 0 \\ -\gamma\phi^{-1} & 1 \end{pmatrix}$ on $\begin{bmatrix} b_2 \\ E \end{bmatrix}$ we see that the complex segments \ref{eq:complex-segment-1} and \ref{eq:complex-segment-2} are isomorphic.
Then, we remove the short complex $ 0 \rightarrow \begin{bmatrix} b_1 \end{bmatrix} \stackrel{\phi}{\rightarrow}\begin{bmatrix} b_2 \end{bmatrix} \rightarrow 0$ which is acyclic.
\end{proof}

Now, we are ready to prove Theorem \ref{theo:hom=hom-simplicial}.

\begin{proof}[Proof of Theorem \ref{theo:hom=hom-simplicial}]
Like for the proof of Theorem \ref{theo:d2=0}, we make an induction on the number of edges belonging to the matching defining the combinatorial Morse vector field.
Let $X$ be a simplicial complex, $K$ be the set of its cells.
\paragraph{\textbf{Initialization:} matching with no edge.}
In this case, there is nothing to prove since the combinatorial Thom-Smale complex is exactly the simplicial chain complex.

\paragraph{\textbf{Heredity:} suppose the property is true for every matching with at most $k$ edges defining a combinatorial Morse vector field.}
Let $V$ be a combinatorial Morse vector field whose underlying matching consists of $k+1$ edges.
Let $\sigma^{(n)}<\tau^{(n+1)}$ be an element of this matching and $\overline{V}$ be the combinatorial Morse vector field equal to $V$ with the matching $\sigma<\tau$ removed (it is actually a combinatorial Morse vector field).
So, $C_i^V = C_i^{\overline{V}}$ for all $i\neq n,n+1$ and $\partial^{V} = \partial^{\overline{V}}$ when restricted to $C_i^V$ for all $i\notin \{n,n+1,n+2\}$.
Moreover, we have the following equalities: $(\partial^{\overline{V}})^{\vert C_{n+1}^V} = (\partial^{V})^{\vert C_{n+1}^V}$ and $\partial^{\overline{V}}_{\vert C_{n}^V} =  \partial^V_{\vert C_{n}^V}$.
By induction hypothesis, the combinatorial Thom-Smale complex $(C^{\overline{V}}_*,\partial^{\overline{V}})$ is a chain complex homotopy equivalent to the simplicial chain complex of $X$.
Thus, the combinatorial Thom-Smale complex associated with $\overline{V}$ is equal to the one of $V$ except on the following chain segment (where $\varepsilon = (\partial^{\overline{V}})_{\vert C^V_{n+1}}^{\vert C_n^V}$):
\begin{equation}\label{eq:complex-segment-Vbar}
\ldots \rightarrow
{ \begin{bmatrix} C^{V}_{n+2} \end{bmatrix} }
\xrightarrow[]{\begin{pmatrix} \alpha \\ \partial^V \end{pmatrix}}
{ \begin{bmatrix} \tau \\ C^V_{n+1} \end{bmatrix} }
\xrightarrow[]{\begin{pmatrix} <\partial \tau,\sigma> & \delta \\ \gamma & \varepsilon \end{pmatrix}}
{ \begin{bmatrix} \sigma \\ C^V_n \end{bmatrix} }
\xrightarrow[]{\begin{pmatrix}\mu & \partial^V \end{pmatrix}}
{ \begin{bmatrix} C^V_{n-1} \end{bmatrix} }
\rightarrow \ldots
\end{equation}

Since $X$ is a simplicial complex we have $<\partial \tau,\sigma>\in \{\pm 1\}$.
Applying lemma \ref{lemma:gaussian-elimination}, we obtain the following new combinatorial chain complex which is homotopic to the combinatorial Thom-Smale complex of $\overline{V}$

\begin{equation}\label{eq:complex-segment-V}
 \ldots \rightarrow
 \begin{bmatrix} C^{V}_{n+2} \end{bmatrix}
 \xrightarrow[]{ \begin{pmatrix}  \partial^V \end{pmatrix} }
 \begin{bmatrix} C^V_{n+1} \end{bmatrix}
 \xrightarrow[]{ \begin{pmatrix} \alpha \end{pmatrix}  }
\begin{bmatrix} C^V_{n} \end{bmatrix}
\xrightarrow[]{ \begin{pmatrix} \partial^V \end{pmatrix} }
\begin{bmatrix} C^V_{n-1} \end{bmatrix}
\rightarrow \ldots
\end{equation}
where $\alpha = \varepsilon - \gamma <\partial \tau,\sigma> \delta = \partial^V$.
Thus, the only thing to prove is that $\alpha = \partial^V$ over $C^V_{n+1}$.
To do this, we investigate $\partial^V$.
There are two types of contributions.
First type correspond to $V$-paths which do not go through $\sigma$, and they are counted in $\varepsilon$.
Second type are $V$-paths which go through $\sigma$.
They begin at an hyperface of a critical cell $\tau'$ and go through $\sigma$: this is the contribution of $\delta$.
Then, they jump to $\tau$: this is the contribution of $<\partial \tau,\sigma>$.
Finally they begin at an hyperface of $\tau$ and go to a critical cell in $ C^V_{n}$: this is the contribution of $\gamma$.
It remains to check that the sign is correct, but this is exactly the same as in the first proof of Theorem \ref{theo:d2=0}.
\end{proof}

\begin{corollary}\label{corollary:algebra-discrete-morse-theory}
 Let $X$ be a finite simplicial complex, $\mathcal C  = (C_* ,\partial)$ be the corresponding simplicial chain complex and $\mathcal M$ be a matching $(\sigma_i < \tau_i)_{i\in I}$ on its Hasse diagram defining a combinatorial vector field $V$. Then the following properties are equivalent:
\begin{enumerate}
 \item $\mathcal M$ is a Morse matching,
 \item for any sequence $(\sigma_{i_1} < \tau_{i_1}),(\sigma_{i_2} < \tau_{i_2}),\ldots, (\sigma_{ i_{\vert I\vert} } < \tau_{ i_{\vert I \vert} })$ such that $i_j\neq i_k$ if $j\neq k$, Gaussian eliminations can be performed in this order.
\end{enumerate}
In particular, any sequence of Gaussian eliminations corresponding to $\mathcal M$ lead to the same chain complex which is the combinatorial Thom-Smale complex of $V$.
\end{corollary}

\begin{proof}
\begin{itemize}
 \item[$1 \Rightarrow 2$] This is an immediate consequence of the proof of Theorem \ref{theo:hom=hom-simplicial} and the fact that it leads to the Thom-Smale complex associated to $(X,V)$.
 \item[$2 \Rightarrow 1$] It is enough to show that there is no non-stationary closed path under the hypothesis.
Suppose there is a closed $V$-path $\gamma:\sigma_{1},\ldots, \sigma_{r} ,\sigma_{1}$ and consider any sequence of Gaussian elimination which coincides with $(\sigma_j < V(\sigma_j))$ until step $r$.
In particular, $r\geq 3$ since $X$ is a simplicial complex.
Let $V$ be the corresponding combinatorial vector field.
Let $\gamma':\sigma_{1},\ldots, \sigma_{r}$ be the $V$-path with length decrease by one.
Then
$$<\partial^{r-1} V(\sigma_{r}),\sigma_r> = <\partial V(\sigma_{r}),\sigma_r> + m(\gamma')$$
Since $m(\gamma')=\pm1$, $<\partial^{r-1} V(\sigma_{r}),\sigma_r>$ is not invertible over $\mathbb{Z}$ and the Gaussian elimination cannot be performed (see lemma \ref{lemma:gaussian-elimination}).
This is a contradiction.
\end{itemize}
\end{proof}

\section[Smooth and discrete Morse theories]{Relation between smooth and discrete Morse theories}\label{section:link-between-discrete-and-smooth}

In this section, we investigate the link between smooth and discrete Morse theories.
We first recall briefly the main ingredients of smooth Morse theory.
In particular, we describe the Thom-Smale complex and prove the following:

\begin{theorem}[Combinatorial realization]\label{theo:combinatorial-realization}
 Let $M$ be a smooth closed oriented Riemannian manifold and $f:M\rightarrow \mathbb{R}$ be a generic Morse function.
Suppose that every stable manifold has been given an orientation so that the smooth Thom-Smale complex is defined.
Then, there exists a $C^1$-triangulation $T$ of $M$ and a combinatorial Morse vector field $V$ on it which realize the smooth Thom-Smale complex (after a choice of orientation of each cells of $T$) in the following sense:
\begin{enumerate}
 \item there is a bijection between the set of critical cells and the set of critical points,
 \item for each pair of critical cells $\sigma_p$ and $\sigma_q$ such that $dim(\sigma_p)=dim(\sigma_q)+1$, $V$-paths from hyperfaces of $\sigma_p$ to $\sigma_q$ are in bijection with integral curves of $v$ up to renormalization connecting $q$ to $p$,
 \item this bijection induce an isomorphism between the smooth and the combinatorial Thom-Smale complexes.
\end{enumerate}
\end{theorem}

Throughout this section, we follow conventions of Milnor (\cite{Milnor1}, \cite{Milnor2}).

\subsection{Smooth Morse theory}
Let $M$ be a smooth closed oriented Riemannian manifold of dimension $n$.
Given a smooth function $f:M\rightarrow \mathbb{R}$, a point $p\in M$ is said critical if $Df(p)=0$.
Let $Crit(f)$ be the set of critical points.
At a critical point $p$, we consider the bilinear form $D^2 f(p)$.
The number of negative eighenvalues of $D^2 f(p)$ is called the index of $p$ (denoted $ind(p)$).
We denote $Crit_k(f)$ the set of critical points of index $k$.

\begin{definition}
 A smooth map $f:M\rightarrow \mathbb{R}$ is called a Morse function if at each critical point $p$ of $f$, $D^2 f(p)$ is non-degenerate.
\end{definition}

More generally, a Morse function on a (smooth) cobordism $(M;M_0,M_1)$ is a smooth map $f:M\rightarrow [a,b]$ such that
\begin{enumerate}
 \item $f^{-1}(a)=M_0$, $f^{-1}(b)=M_1$,
 \item all critical points of $f$ are interior (lie in $M-(M_0 \cup M_1)$) and are non-degenerate.
\end{enumerate}

For technical reasons, we must consider the following object:
\begin{definition}
Let $f$ be a Morse function on a cobordism $(M^n;M_0,M_1)$.
A vector field $v$ on $M^n$ is a gradient-like vector field for $f$ if
\begin{enumerate}
 \item $v(f)>0$ throughout the complement of the set of critical points of $f$,
 \item given any critical point $p$ of $f$ there is a Morse chart in a neighbourhood $U$ of $p$ so that
$$
f(x) = f(p) - \sum_{i=1}^{k}x_i^2 + \sum_{i=k+1}^{n}x_i^2
$$
 and $v$ has coordinates $v(x) = (-x_1,\ldots,-x_k,x_{k+1},\ldots,x_n)$.
\end{enumerate}
\end{definition}

Given any Morse function, there always exists a  gradient-like vector field (see \cite{Milnor2}).
In the following, we shall abreviate ``gradient like vector field'' by ``\emph{gradient}''.
Thus, when needed, we will assume that we have chosen one.

Given any $x_0\in M$, we consider the following Cauchy problem
\begin{center}
$\left\lbrace 
\begin{tabular}{rcl}
 $\gamma'(t)$ & $=$ & $v(\gamma(t))$ \\
$\gamma(0)$ & $=$ & $x_0$ \\
\end{tabular}
\right.$ 
\end{center}
and call integral curve (denoted $\gamma_{x_0}$) the solution of this Cauchy problem.
The stable manifold of a critical point $p$ is by definition the set $W^s (p,v):=\{x\in M \vert \lim\limits_{t\rightarrow +\infty} \gamma_x(t)=p\}$.
The unstable manifold of a critical point $p$ is by definition the set $\{x\in M \vert \lim\limits_{t\rightarrow -\infty} \gamma_x(t)=p\}$.
When stable and unstable manifolds are transverse (this is called Morse-Smale condition), we called $v$ a Morse--Smale gradient: such gradient always exists in a neighbourhood of a gradient (see e.g. \cite{pajitnov1}).
We shall call a Morse function $f$ generic if we have chosen for $f$ a Morse--Smale gradient.

To define the smooth Thom-Smale complex we need the following data:
\begin{enumerate}
 \item[$\bullet$] a generic Morse function $f$,
 \item[$\bullet$] an orientation of each stable manifold.
\end{enumerate}

Under these conditions, the number of integral curves of $v$ up to renormalization (that is $\gamma_x \sim \gamma_y$ iff there exists $t\in \mathbb{R}$ such that $\gamma_x(t) = y$) connecting two critical points of consecutive index is finite.
Moreover, when we consider an integral curve from $q$ to $p$ where $ind(p)=ind(q)+1$, it carries a coorientation induced by the orientation of the stable manifold and the orientation of the integral curve.
One can move this coorientation from $p$ to $q$ along the integral curve and compare it with the orientation of the stable manifold of $q$.
This gives the sign which is carried by the integral curve connecting $q$ to $p$.

The Thom-Smale complex $(C^f_*,\partial^f)$ is defined as:
\begin{enumerate}
 \item $C^f_k = \bigoplus_{p\in Crit_k(f)} \mathbb{Z}.p$,
 \item if $p \in Crit_k(f)$ then $\partial p = \sum_{q\in Crit_{k-1}(f)} n(p,q).q$ where $n(p,q)$ is the algebraic number of integral curves up to renormalization connecting $q$ to $p$.
\end{enumerate}

\begin{theorem}
 The homology of the Thom-Smale complex is equal to the singular homology of $M$.
\end{theorem}

The proof of this theorem can be extracted from \cite{Milnor2}.

\subsection{Elementary cobordisms}
In this subsection, we will prove that we can realize combinatorially the smooth Thom-Smale complex of any elementary cobordisms.
Thus, by cutting the manifold $M$ into elementary cobordism we will obtain the first part of Theorem \ref{theo:combinatorial-realization}: there exists a bijection between the set of critical cells and the set of critical points.

We will only consider $C^1$-triangulation of manifolds for technical reasons (see \cite{Whitehead}).
So, whenever we use the word triangulation it means $C^1$-triangulation.
A triangulation of a $n+1$-cobordism $(M^{n+1};M_0,M_1)$ is a triplet $(T;T_0,T_1)$ such that $T$ is a $C^1$-triangulation of $M$, $T_0$ (resp. $T_1$) is a subcomplex of $T$ which is a $C^1$-triangulation of $M_0$ (resp. $M_1$).

A combinatorial Morse vector field $V$ on a triangulated $n+1$-cobordism $(T;T_0,T_1)$ is a combinatorial Morse vector field on $T$ such that no cells of $T_1$ is critical and every cell of $T_0$ is critical.

\begin{definition}\label{def:ancestor-property}
 Let $V$ be a combinatorial Morse vector field on a triangulated $n+1$-cobordism $(T;T_0,T_1)$.
$V$ satisfies the ancestor's property if given any $n$-cell $\sigma_0 \in T_0$, there exists an $n$-cell $\sigma_1 \in T_1$ and a $V$-path starting at $\sigma_1$ and ending at $\sigma_0$.
\end{definition}

\begin{remark}
There is a key difference between integral curves up to renormalization of a gradient $v$ and $V$-paths.
Given a point $x\in M$, there is only one solution to the Cauchy problem.
Moreover, the past and the future of a point pushed along the flow is uniquely determined.
\emph{A contrario} given a cell $\sigma$, there are (in general) many $V$-paths starting at $\sigma$.
The ancestor's property caracterises $n+1$-cobordism equipped with a combinatorial Morse vector field which knows its history in maximal dimension $(n,n+1)$.
\end{remark}

To prove that elementary cobordism can be realized , we need a combinatorial description of being a deformation retract.
Let $X$ be a simplicial complex and $\sigma$ be an hyperface of $\tau$ which is free (that is $\sigma$ is a face of no other cell).
In this case, we say that $X$ collapses to $X-(\sigma\cup \tau)$ by an elementary collapsing and write $X \searrow X-(\sigma\cup \tau)$.
A collapsing is a finite sequence of such elementary collapsings.
In particular, a collapsing defines a matching on the Hasse diagram of the simplicial complex.
Moreover, one can prove that $X-(\sigma\cup\tau)$ is a deformation retract of $X$.

\begin{proposition}
 Let $X$ be a simplicial complex and $X_0$ be a subcomplex.
Suppose $X\searrow X_0$.
Then the matching given by this collapsing defines a combinatorial Morse vector field whose set of critical cells is the set of cells of $X_0$.
\end{proposition}

\begin{proof}
 The only thing to check is that there is no non-stationary closed path.
Since elementary collapsings are performed by choosing a free hyperface of a cell, there is no non-stationary closed path.
\end{proof}

Let $\Delta^m = (a_0,\ldots,a_m)$ be the standard simplex of dimension $m$.
The cartesian product $X=\Delta^m\times\Delta^n$ is the cellular complex whose set of cells is $\{\mu\times \nu\}$ where $\mu$ (resp. $\nu$) is a cell of $\Delta^m$ (resp. $\Delta^n$) (see \cite{Whitney1}).

\begin{proposition}[{\cite[Prop. 2.9]{RourkeSanderson}}]\label{proposition:cartesian-product}
 The cartesian product $\Delta^m \times \Delta^n$ has a simplicial subdivision without any new vertex.
More generally, the cartesian product of two simplicial complexes has a simplicial subdivision without any new vertex.
\end{proposition}

\begin{lemma}\label{lemma:product-of-simplex}
 Let $X_1=\Delta^k$ be the standard simplicial complex of dimension $k$ and $X_0$ be a simplicial subdivision of $X_1$.
Consider the CW-complex which is equal to the cartesian product $\Delta^k \times\Delta^1$ and where we subdivide $\Delta^k\times \{0\}$ so that it is equal to $X_0$.
Then, there exists a simplicial subdivision $X$ of this CW-complex such that $X_{\vert \Delta \times \{i\} } = X_i$ for $i\in \{0,1\}$.

Moreover for $i\in \{0,1\}$ there exists a collapsing $X \searrow X_i$ and the combinatorial Morse vector field associated $V_i$ satifies the ancestor's property on $(X;X_{\overline{i}},X_{\overline{i+1}})$ ($\overline{j}$ is the class in $\mathbb{Z}/2\mathbb{Z}$).
\end{lemma}

\begin{proof}
 The simplicial subdivision and the collapsing is constructed by induction on $k$.
If $k=0$, choose a new vertex in the interior of the simplex $\Delta^0\times  \Delta^1$ and the elementary collapsing $\Delta^1 \searrow \{i\}$ gives the two collapsing $\Delta^0 \times \Delta^1 \searrow \Delta^0 \times \{i\}$ for $i\in \{0,1\}$.
In particular, the corresponding combinatorial Morse vector field satifies the ancestor's property.

Suppose the lemma is true until rank $k-1$.
At rank $k$, let $Y$ be the corresponding CW-complex and $x$ be a point in the interior of the cell of dimension $k+1$.
By induction hypothesis $Y_{\vert \partial \Delta^k \times \Delta^1}$ admits a simplicial subdivision.
Therefore, $Y_{\vert \partial(\Delta^k \times \Delta^1) }$ admits a simplicial subdivision denoted $Z$ (just add the simplexes $\Delta^k \times \{i\}$ which are equal to $X_i$ for $i\in\{0,1\}$).
The simplicial subdivision $X$ is given by making the join of the simplicial subdivision of the boundary over $\{x\}$: $X = Z * \{x\}$.

Now, the collapsing $X \searrow X_0$ is performed in three steps.
\begin{itemize}
 \item[\textbf{Step 1.}] The cell $\sigma \in X_1$ of dimension $k$ is the free hyperface of the cell $\sigma *\{ x \}$.
We do the following elementary collapsing:
\begin{equation}\label{eq:collapsing-1-product-of-simplex}
 X \searrow X -(\sigma \cup \sigma*\{x\})
\end{equation}
 \item[\textbf{Step 2.}] By induction hypothesis, $X_{ \vert \partial \Delta^k \times \Delta^1 } \searrow X_{ \vert \partial \Delta^k \times \{0\} }$.
Performing the join over $x$ induces the following collapsing:
\begin{equation}\label{eq:collapsing-2-product-of-simplex}
 X_{ \vert \partial \Delta^k \times \Delta^1 }*\{x\}  \searrow X_{ \vert \partial \Delta^k \times \{0\} } *\{x\}
\end{equation}
 \item[\textbf{Step 3.}] It remains to collapse $X_0 *\{x\}$ on $X_0$.
Let $y$ be a vertex in $X_0$ which is a vertex of the original simplex $X_1$.
Since $X_0$ is a simplicial subdivision of $\Delta^k$, there exists a collapsing $X_0 \searrow \{y\}$.
This collapsing gives the following collapsing:
\begin{equation}\label{eq:collapsing-3-product-of-simplex}
 X_0 * \{x\} \searrow X_0 \cup (\{y\}\times \{0\})* \{x\} \searrow X_0
\end{equation}
\end{itemize}

Combining collapsings \eqref{eq:collapsing-1-product-of-simplex}, \eqref{eq:collapsing-2-product-of-simplex} and \eqref{eq:collapsing-3-product-of-simplex} gives $X\searrow X_0$.
The corresponding combinatorial Morse vector field satisfies the ancestor's property by construction.

The collapsing $X \searrow X_1$ is constructed in the same way and conclusions of lemma follows.
\end{proof}

\begin{remark}
The proof of lemma \ref{lemma:product-of-simplex} is by induction.
Let $\delta^{(j)}$ be the $j$-th skeleton of $\Delta^k$.
Denote by $X^{(j)}$ (resp. $X^{(j)}_i$) the simplicial complex $X_{\vert \delta^{(j)} \times \Delta^1}$ (resp. $(X_{i})_{\vert \delta^{(j)} \times \Delta^1}$).
For $i\in \{0,1\}$, the collapsing $X\searrow X_i$ can be restricted to $X^{(j)}\searrow X^{(j)}_i$ for any $0\leq j \leq k$ and the induced combinatorial Morse vector field satisfies the ancestor's property.
\end{remark}

The next two lemmas are technical lemmas.
The first one is the basic tool to glue together triangulated cobordisms.
The second one will be useful to construct a combinatorial realization of a cobordism with exactly one critical point and is a generalization of lemma \ref{lemma:product-of-simplex}.

\begin{lemma}\label{lemma:gluing-cobordisms}
 Let $(T_i^M,T_i^N)$ be two $C^1$-triangulations of the pair $(M,N)$ where $N^{k}$ is a submanifold (possibly with boundary) of $M^n$ ($k\leq n$).
Then, there exists a $C^1$-triangulation $T$ of $(M\times [0,1],N\times[0,1])$ such that
$$(T_{ \vert M\times\{i\} } , T_{ \vert N\times\{i\} }) = (T_i^M,T_i^N)$$
for $i\in \{0,1\}$ and 2 collapsings
\begin{equation}
 T \searrow T_0^M \cup T_{\vert N\times [0,1]}
\end{equation}
\begin{equation}
 T_{\vert N\times[0,1]} \searrow T_0^N
\end{equation}
Moreover, the induced combinatorial Morse vector fields $V$ satifies the ancestor's property on the cobordisms $(T_{\vert N\times[0,1]};T_0^N,T_1^N)$ and $ (T;T_0^M,T_1^M) $.
\end{lemma}

\begin{proof}
First, suppose $N=\emptyset$.
Both triangulations $T_0$ and $T_1$ are $C^1$-triangulation of the same manifold therefore they have a common simplicial subdivision $T^{1/2}$ \cite{Whitehead} (this is where we use the fact that triangulations are $C^1$-triangulations).
Subdivide $\Delta^1 = [0,1]$ in two standard simplexes $[0,1/2]$ and $[1/2,1]$.
Lemma \ref{lemma:product-of-simplex} gives a $C^1$-triangulation of $M\times[0,1/2]$ (resp. $M\times[1/2,1]$) denoted $T^{[0,1/2]}$ (resp. $T^{[1/2,1]}$).
The union $T^{[0,1/2]}\cup T^{[1/2,1]}$ is a triangulation of $M\times[0,1]$ denoted $T$.
By construction, $T_{\vert M\times \{i\} } = T_i^M$ for $i\in \{0,1\}$ and we have the two following collapsings:
\begin{equation*}
 T^{[1/2,1]}\searrow T^{1/2}
\end{equation*}
\begin{equation*}
 T^{[0,1/2]}\searrow T^{0}
\end{equation*}
Composing these two collapsings give the desired collapsing and lemma \ref{lemma:product-of-simplex} give the ancestor's property.

In the case where the submanifold $N$ is non-empty, the construction above gives a triangulation $T$ of the pair $(M\times[0,1],N\times[0,1])$ and we have $(T_{\vert M\times\{i\} },T_{\vert N\times\{i\} }) = (T_i^M,T_i^N)$ for $i\in \{0,1\}$.
The collapsing $T\searrow T_0$ can be restricted to $T_{\vert N\times[0,1]}$.
We remove from the matching edges corresponding to $T_{N\times[0,1]}\searrow T_{N\times\{0\} }$ to obtain the desired collapsing.
Again lemma \ref{lemma:product-of-simplex} give the ancestor's property.
\end{proof}

\begin{lemma}\label{lemma:product-simplex-nxm}
 Let $(m,n)$ be a pair of positive integers.
Let $\Delta^n = (a_0,\ldots,a_n)$ be the standard simplex of dimension $n$ and $\delta^{n-1} = (\widehat{a}_0 ,\ldots,a_n)$ be the hyperface which does not contain $a_0$.
In particular $\Delta^n =  \{a_0\}*\delta^{n-1}$.
Then there exists a simplicial subdivision $X$ of the cartesian product $\Delta^m \times \Delta^n$ such that
\begin{itemize}
 \item[$\bullet$] $X_{\vert \Delta^m \times\delta^{n-1} }$ is a simplicial subdivision without any new vertex given by lemma \ref{proposition:cartesian-product},
 \item[$\bullet$] $X_{\vert \Delta^m \times\{a_0\} } = \Delta^m$,
 \item[$\bullet$] $X \searrow X_{ \vert( \partial \Delta^m \times \Delta^n ) \cup (\Delta^m \times\{a_0\}) }$.
\end{itemize}
Moreover, for each simplex $\Delta^1_i = (a_0,a_i)$ ($i\neq 0$),
\begin{itemize}
 \item[$\bullet$] $X_{\vert \Delta^m \times \Delta^1_i}$ coincides with the simplicial complex given by lemma \ref{lemma:product-of-simplex},
 \item[$\bullet$] the collapsing  $X \searrow X_{ \vert( \partial \Delta^m \times \Delta^n ) \cup (\Delta^m \times\{a_0\}) }$ restricted to $X_{\vert \Delta^m \times \Delta^1_i}$ coincides with the collapsing of lemma \ref{lemma:product-of-simplex},
 \item[$\bullet$] the induced combinatorial Morse vector field satisfies the ancestor's property on $(X_{\vert \Delta^m \times \Delta^1_i}; X_{\vert \Delta^m \times\{a_0\} } , X_{\vert \Delta^m \times\{a_i\} })$.
\end{itemize}
\end{lemma}

\begin{proof}
The proof is by induction on $k=m+n>0$.
At rank $k=1$ there are two cases.
The case $m=0$ and $n=1$ is trivial: there is nothing to prove.
The case $m=1$ and $n=0$ is given by lemma \ref{lemma:product-of-simplex}.

Suppose the lemma is true until rank $k-1$.
Let $(m,n)\in \mathbb{N}^2$ be such that $m+n=k$.
We will first subdivide the boundary of $\Delta^m \times \Delta^n$.
Since
$$
\begin{array}{rcl}
  \partial (\Delta^m \times \Delta^n ) & = & (\partial \Delta^m \times\Delta^n ) \cup (\Delta^m \times \partial \Delta^n) \\
& = & (\partial \Delta^m \times \Delta^n ) \cup ( \Delta^m \times (\{a_0\} * \partial \delta^{n-1}))   \cup (    \Delta^m \times \delta^{n-1}  )  \\
\end{array}
$$
we define for each cellular complex above a simplicial subdivision.

\begin{itemize}
 \item[$\bullet$] The simplicial subdivision of $\Delta^m \times \delta^{n-1}$ is given by Proposition \ref{proposition:cartesian-product}: in particular, we do not create any new vertex.
 \item[$\bullet$] The induction hypothesis gives a simplicial subdivision of
$$
(\partial \Delta^m \times (\{a_0\} * \delta^{n-1}) ) \cup ( \Delta^m \times (\{a_0\} * \partial \delta^{n-1}) ).
$$
\end{itemize}

Let $x$ be a point which is in the interior of the $(m+n)$-cell of $\Delta^m \times \Delta^n$.
The simplicial subdivision $X$ of $\Delta^m \times \Delta^n$ is given by making the cone over $\{x\}$ of the simplicial subdivision of the boundary of $\Delta^m \times \Delta^n$.

By construction we have the following collapsing
\begin{equation}\label{eq:collapsing-upper-boundary}
  X_{\vert \{x\}* (\Delta^{m}\times \delta^{n-1}) } \searrow X_{\vert \{x\}* \partial (\Delta^{m}\times \delta^{n-1}) }
\end{equation}
which is realized by a downward induction on the dimension of cells of $\Delta^m \times (\delta^{n-1}-\partial \delta^{n-1})$: every cell $\sigma \in \Delta^m \times (\delta^{n-1}-\partial \delta^{n-1})$ is a free hyperface of $\{x\} *\sigma$.

The induction hypothesis says that there exits a simplicial subdivision $Y$ of $\Delta^u \times \Delta^v$ such that $Y\searrow Y_{\vert (\partial \Delta^u \times  \Delta^v )\cup( \Delta^u \times \{a_0\} ) }$ whenever $u+v<k$ ($a_0$ is the first vertex of $\Delta^v$).
In fact, we have also the following collapsing since the construction is made by induction:
$$
Y \searrow Y_{\vert \Delta^u \times \{a_0\}  }
$$
Therefore, we have the following collapsings
\begin{equation}\label{eq:collapsing-boundary}
  X_{\vert \partial \Delta^m \times (\{a_0\} * \delta^{n-1}) } \searrow X_{\vert \partial \Delta^m \times \{a_0\}  }
\end{equation}
\begin{equation}\label{eq:collapsing-lower-boundary}
  X_{\vert \Delta^m \times \{a_0\} * \partial \delta^{n-1} } \searrow X_{\vert \Delta^m \times \{a_0\} }
\end{equation}

Collapsing \eqref{eq:collapsing-upper-boundary} followed by the cone over $x$ of the collapsing \eqref{eq:collapsing-boundary} and the cone over x of the collapsing \eqref{eq:collapsing-lower-boundary} give the following collapsing:
$$
X\searrow X_{\vert (\partial \Delta^m \times  \Delta^n )\cup( \{x\} * (\Delta^m \times  \{a_0\})   ) }
$$
Finally there exists a collapsing $\{x\}*(\Delta^m \times \{a_0\} )  \searrow \Delta^m\times \{a_0\}$ (by choosing a vertex $y\in \Delta^m$ and considering the collapsing $\Delta^m \searrow \{y\}$) which gives the result.

In case $n=1$, this construction is the same as the one of lemma \ref{lemma:product-of-simplex}.
\end{proof}

\begin{theorem}\label{theo:cobordism-one-critical-point}
 Let $f$ be a generic Morse function on a cobordism $(M;M_0,M_1)$ with exactly one critical point $p$ of index $k$.
Then, there exists a $C^1$-triangulation of the cobordism $(T;T_0,T_1)$ such that
\begin{enumerate}
 \item the stable manifold of $p$ is a subcomplex of $T$ denoted $T^s_p$ and $T\searrow T^s_p \cup T_0$,
 \item there is a cell $\sigma_p$ of dimension $ind(p)$ such that $p\in \sigma_p \subset T^s_p$ and
$T^s_p - \sigma_p \searrow (T^s_p \cap T_0 )$
\end{enumerate}
In particular, the combinatorial Morse vector field given by these two collapsings has exactly one critical cell $\sigma_p$ outside cells of $T_0$.
\end{theorem}

\begin{proof}
Suppose a Morse--Smale gradient $v$ for $f$ is fixed.
Let $W^s(p,v)$ be the corresponding stable manifold of $p$.
We follow the proof of Milnor which proves that $M_0 \cup W^s(p,v)$ is a deformation retract of $M$ (see the proof of Theorem 3.14 \cite{Milnor1}).
Let $C$ be a (small enough) tubular neighbourhood of $W^s(p,v)$.
The original proof consists of two steps.
First, $M_0 \cup C$ is a deformation retract of $M$: this is done by pushing along the gradient lines of $v$.
Then, $M_0 \cup W^s(p,v)$ is a deformation retract of $M_0 \cup C$.
We prove the theorem in two steps.

\paragraph*{\textbf{First step: construction of a good triangulation of $C$.}}\hfill\\
The tubular neighbourhood $C$ is diffeomorphic to $D^k \times D^{n-k}$ (for $i\in \mathbb{N}^*$, $D^i$ is the unit disk in $\mathbb{R}^i$).
Thanks to this diffeomorphism, the stable manifold is identified with $D^k \times \{0\}$ and the adherence of the unstable manifold is identified with $\{0\}\times D^{n-k}$.
Triangulate the stable manifold by the standard simplex $\Delta^k$ and denote $\sigma_p$ its interior (so $T^s_p = \overline{\sigma}_p$).

We triangulate $D^{n-k}$ by choosing an arbitrary triangulation of $\partial D^{n-k} = S^{n-k-1}$ and considering $D^{n-k}$ as the cone over its center $\{0\}$: this gives a triangulation of $D^{n-k}$.
The triangulation of $\overline{\sigma}_p\times D^{n-k}$ is the following one: choose a simplicial subdivision of $\overline{\sigma}_p\times \partial D^{n-k}$ without any new vertex given by proposition \ref{proposition:cartesian-product}.
Then, triangulate the cartesian product $\overline{\sigma}_p\times D^{n-k}$ with the triangulation of $\overline{\sigma}_p\times \partial D^{n-k}$ already fixed thanks to lemma \ref{lemma:product-simplex-nxm}:
\begin{itemize}
 \item[$\bullet$] for each simplex $\nu\in \partial D^{n-k}$, the lemma constructs a triangulation of $\overline{\sigma}_p  \times (\{0\}*\nu)$,
 \item[$\bullet$] for each pair of simplexes $(\nu_0,\nu_1) \in (\partial D^{n-k})^2$, the simplicial subidivisions of $\overline{\sigma}_p \times (\{0\}*\nu_i)$ coincides over $\overline{\sigma}_p \times (\{0\}*(\nu_0 \cap \nu_1))$.
\end{itemize}

Let $T^C$ be the triangulation of $\overline{\sigma}_p \times D^{n-k}$ constructed above.
By construction, we have the following collapsing

\begin{equation}\label{eq:collapsing-over-critical-cell}
  T^C \searrow T^s_p \cup T^C _{ \vert \partial \overline{\sigma}_p \times D^{n-k} }
\end{equation}

\paragraph*{\textbf{Second step: combinatorial realization of the first retraction.}}\hfill\\
Let $T_0$ be a triangulation of $M_0$ which coincides over $M_0 \cap C$ with the triangulation above. 
Consider the following submanifolds with boundary: $\partial C_- = M_0 \cap C$, $M_0^{\overline C} = M_0 - Int(\partial C_-)$ and $\partial C_+ = \partial C - Int(\partial C_-)$.
Let $V$ be $\partial C_- \cap \partial C_+$ : it is diffeomorphic to $\partial D^k \times \partial D^{n-k} = S^{k-1} \times S^{n-k-1}$ a manifold of dimension $n-2$.

The manifold $(\partial C_+ ,V)$ is a manifold with boundary which is triangulated.
The gradient lines of $v$ starting at any point of this manifold are transverse to it: we push along the gradient lines of $v$ the triangulation until it meets $M_1$.
It gives a triangulation of $(M_1^{\partial C_+ },M_1^V)$ which is a submanifold of $M_1$ with boundary.
This triangulation is $C^1$ since pushing along the flow in this case is a diffeomorphism.
Then, we get a product cobordism (with boundary) with triangulation of the top and the bottom already fixed: lemma \ref{lemma:gluing-cobordisms} gives a triangulation of this cobordism with the desired collapsing.

The same construction holds for  $(M_0^{\overline C},V)$ (we suppose that the triangulation of $V\times [0,1]$ is the same as the one given above).
Let $T$ be the corresponding triangulation of $M$.
Then, we have the following collapsing
\begin{equation}\label{eq:collapsing-over-tubular-neighbourhood}
  T \searrow T_{0} \cup T^C
\end{equation}

\paragraph*{\textbf{Conclusion.}} The composition of collapsings \eqref{eq:collapsing-over-tubular-neighbourhood} and \eqref{eq:collapsing-over-critical-cell} give
$$
T \searrow T_0 \cup T^s_p
$$
Since $T^s_p = \overline{\sigma}_p$ we get the following collapsing: $T^s_p - \sigma_p \searrow \partial T^s_p$. Thus a combinatorial Morse vector field which satifies the conclusion of the theorem has been constructed.
Nevertheless, note that the triangulation above in not $C^1$: the triangulation of the stable manifold done by $\Delta^k$ gives only a topological triangulation.
To correct this, push the level $M_0$ (denote this level $M_0'$) along the gradient line a little inside the cobordism so that the stable manifold can be $C^1$-triangulated by the standard simplex.
Then, we endow the cobordism whose boundary is $M_0 \cup M_0'$ with a $C^1$-triangulation given by Lemma \ref{lemma:product-of-simplex}.
\end{proof}

\begin{corollary}\label{cor:bijection-between-critical-cells-and-points}
 Let $f$ be a generic Morse function on a Riemannian closed manifold $M$.
Then, there exists $T$ a $C^1$-triangulation of $M$ and a combinatorial Morse vector field $V$ defined on $T$ such that for every $k\in \mathbb{N}$ the set of critical poins of index $k$ is in bijection with the set of critical cells of dimension $k$.
\end{corollary}

\begin{proof}
 Since the Morse function $f$ is generic, we have that for any critical points $p\neq q$, $f(p)\neq f(q)$.
Let $a_1<a_2<\ldots < a_l$ be the ordered set of critical values of $f$.
For each $k\in \{1,\ldots,l\}$, let $\varepsilon_k>0$ be small enough so that the cobordism
$$
( M^{a_k};M^{a_k}_-,M^{a_k}_+ ) =\big( f^{-1}([a_k-\varepsilon_k,a_k+\varepsilon_k]);f^{-1}(a_k-\varepsilon_k),f^{-1}(a_k+\varepsilon_k)
\big)
$$
is a cobordism with exactly one critical point.
Define for $k\in \{1,\ldots,l-1\}$ the product cobordisms
$$
(M^{b_k};M^{a_{k-1}}_+ , M^{a_k}_- ) = \big( f^{-1}([a_{k-1}+\varepsilon_{k-1},a_k-\varepsilon_k]);f^{-1}(a_{k-1}+\varepsilon_{k-1}),f^{-1}(a_k-\varepsilon_k) \big)
$$
The manifold $M$ is equal to:
$$M^{a_1}  \cup M^{b_1} \cup \ldots \cup M^{b_{l-1}} \cup M^{a_l}$$
Theorem \ref{theo:cobordism-one-critical-point} gives for $k=1,\ldots,l$ a combinatorial realization of the cobordism $( M^{a_k};M^{a_k}_-,M^{a_k}_+ )$.
Lemma \ref{lemma:gluing-cobordisms} gives a combinatorial realization of each cobordism $(M^{b_k};M^{a_{k-1}}_+ , M^{a_k}_- )$ for $k=1,\ldots,l-1$ (with the convention that $M^{a_0} = \emptyset$).
Then, we construct a $C^1$-triangulation of $M$ and define on it a combinatorial vector field.
It is in fact a combinatorial Morse vector field since along $V$-paths we only can go down and the conclusion of the corollary follows.
\end{proof}

\subsection{Proof of Theorem \ref{theo:combinatorial-realization}}
Since $f$ is generic, we use the Rearrangement Theorem \cite[Theorem 4.8]{Milnor2} to consider $g$ a generic self-indexed Morse function such that
\begin{itemize}
 \item[$\bullet$] the set of critical points of index $k$ of $g$ coincides with the one of $f$ for every $k\in \mathbb{N}$,
 \item[$\bullet$] for each pair of critical points $p$ and $q$ of successive index, the set of integral curves up to renormalization connecting $q$ to $p$ for $g$ is in bijection with the corresponding set for $f$ (we suppose here that Morse--Smale gradients have been chosen for $f$ and for $g$),
 \item[$\bullet$] this bijection induces an isomorphism between the Thom-Smale complexes of $f$ and $g$ (we suppose that orientations of stable manifolds have been chosen).
\end{itemize}

Thus, we suppose that $f:M^n\rightarrow \mathbb{R}$ is a generic self-indexed Morse function i.e. for every $k\in \mathbb{N}$, for every $p\in Crit_k(f)$, $f(p)=k$.
In particular $f(M)=[0,n]$.
We suppose whenever we need it that a Morse--Smale gradient $v$ for $f$ is given.

One more time, we will cut $M$ in cobordisms (almost) elementary and control combinatorially the behavior of $V$-paths.
For $i\in \{0,\ldots,n\}$ choose $0<\varepsilon_i<1/2$.
For $i\in \{0,\ldots,n\}$, let $(M^i;M^{i}_{-,},M^{i}_{+})$ be the cobordism
\begin{equation*}
(f^{-1}([i-\varepsilon_i,i+\varepsilon_i]); f^{-1}(i- \varepsilon_i), f^{-1}(i+ \varepsilon_i) )
\end{equation*}
Similarly, define $(M^{i,i+1};M^{i}_{+},M^{i+1}_{-})$ the product cobordism equal to
\begin{equation*}
(f^{-1}([i+\varepsilon_i,i+1-\varepsilon_{i+1}]) ; M^{i}_{+},M^{i+1}_{-})
\end{equation*}
Then
\begin{equation*}
 M =  M^0 \cup M^{0,1} \cup M^1 \cup \ldots \cup M^{n-1,n} \cup M^n
\end{equation*}
For all $i\in\{0,\ldots,n\}$, $(M^{i};M^{i}_{-},M^{i}_{+})$ is a cobordism with $\vert Crit_i(f) \vert$ critical points of index $i$ (maybe there is no critical point).
The triangulation of $M$ is constructed in the following way:
\begin{enumerate}
 \item triangulation of cobordisms $(M^{i};M^{i}_{-},M^{i}_{+})$ for all $i\in \{0,\ldots,n\}$ given by Theorem \ref{theo:cobordism-one-critical-point},
 \item triangulation of cobordisms $(M^{i,i+1};M^{i}_{+},M^{i+1}_{-})$ for all $i\in \{0,\ldots,n\}$ given by lemma \ref{lemma:gluing-cobordisms}.
\end{enumerate}

\begin{remark}
 Theorem \ref{theo:cobordism-one-critical-point} is proved in the case where there is exactly one critical point.
This proof extends directly to the case of $k$ critical points of the same index under the condition that tubular neighbourhoods of stable manifolds are chosen to be disjoints one from each other.
\end{remark}

Let $p$ be a critical point of index $k$ and $C(p)$ be a tubular neighbourhood (small enough) of the stable manifold of $p$ in the corresponding cobordism.
Denote $\partial C_- (p)$ (resp. $\partial C_+ (p)$) the submanifold diffeomorphic to $\partial D^k \times D^{n-k}$ (resp. $D^k \times \partial D^{n-k}$).
Denote $\sigma_p$ the critical cell of dimension $k$ corresponding to $p$ (see Theorem \ref{theo:cobordism-one-critical-point}).

\textbf{Hypothesis on the triangulation of $\partial C_+ (p)$.}
\begin{enumerate}
 \item stable manifolds of critical points of index $k+1$ intersect $\partial C_+ (p)$ along a subcomplex of dimension $k$ and intersect $\sigma_p \times \partial D^{n-k}$ along cells of dimension $k$ of the type $\sigma\times \{a_i\}$ where $a_i$ is a vertex of $\partial D^{n-k}$,
 \item each integral curve up to renormalization $\gamma$ from $p$ to $q\in Crit_{k+1}(f)$ intersects $\partial C_+ (p)$ in the interior of a $k$-cell $\sigma_{\gamma}\in\sigma_p \times \partial D^{n-k}$,
 \item given two distinct integral curves up to renormalization $\gamma$ and $\gamma'$ from $p$ to critical points of index $k+1$ then $\sigma_{\gamma} \neq \sigma_{\gamma'}$.
\end{enumerate}

\begin{remark}
The first hypothesis is satisfied by choosing small enough $\varepsilon_k$ and since stable and unstable manifolds are transverse.
For such a small enough $\varepsilon_k$ the last hypothesis will be satisfied.
The second hypothesis is automatically satisfied if the first hypothesis is satisfied.
\end{remark}

In each triangulated cobordism $(M^{k};M^{k}_{-},M^{k}_{+})$, stable manifolds of critical points of index $k$ are subcomplexes.
Following notations of Theorem \ref{theo:cobordism-one-critical-point}, we have the following collapsing
$$T^s_p - \sigma_p \searrow \partial T^s_p$$
Using lemma \ref{lemma:gluing-cobordisms}, we obtain the following collapsing
$$
M^{k-1,k} \searrow M^{k-1}_+
$$
which can be restricted to the stable manifold of $p$ since it is a submanifold of $M^{k}_{-}$.
With respect to the stable manifold, the combinatorial Morse vector field satisfies the ancestor's property.

Let $\gamma$ be an integral curve of $v$ up to renormalization from $q\in Crit_{k-1}(f)$ to $p$.
It intersects $\partial C_+ (q)$ in a point which by hypothesis belongs to a cell $\sigma_q \times \{a_{\gamma} \}$.
There is a $1-1$ correspondance between the set of integral curves up to renormalization from $q$ to $p$ (with $ind(p)=ind(q)+1$) and $V$-paths from hyperfaces of $\sigma_p$ to $\sigma_q$ given by $\gamma \longleftrightarrow \sigma_{\gamma}$.

From $\sigma_{\gamma}$, there is a unique $V$-path ending at $\sigma_q$.
Since $V$ satisfies the ancestor's property in the stable manifold and $\sigma_{\gamma}$ is a cell of dimension $k-1$ there is an ancestor of $\sigma_{\gamma}$ which is an hyperface of $\sigma_p$.
This gives a $V$-path between an hyperface of $\sigma_p$ to $\sigma_q$ which corresponds to $\gamma$.

We endow each critical cell with the orientation of the corresponding stable manifold and every other cell is endowed with an arbitrary orientation.

By construction, the multiplicity of $V$-path coincides with the sign of the corresponding gradient path and the theorem follows.

\section{Complete matchings and Euler structures}\label{section:complete-matchings-and-euler-structures}
In this section, we use Theorem \ref{theo:combinatorial-realization} to prove the following: given a closed oriented 3-manifold and an Euler structure on it, there is a triangulation such that a complete matching on the Hasse diagram of a triangulation realizes this Euler structure.

\subsection{Complete matchings}
\begin{definition}
A complete matching on a graph is a matching such that every vertex belongs to an edge of the matching.
\end{definition}

As a corollary of theorem \ref{theo:combinatorial-realization} we obtain:

\begin{corollary}\label{cor:existence-matching}
Let $M$ be a closed smooth manifold of dimension 3.
Then there exits a $C^1$-triangulation of $M$ such that a complete matching on its Hasse diagram exists.
\end{corollary}

\begin{proof}
Since $M$ is a closed smooth manifold of dimension 3 we have $\chi(M)=0$ where $\chi$ denotes the Euler characteristic.
Take a pointed Heegaard splitting of $M$ $(\Sigma_g ; \underline{\alpha} = (\alpha_1,\ldots,\alpha_g), \underline{\beta} = (\beta_1,\ldots,\beta_g);z)$ of genus $g$ so that there is an $n$-uplets of intersection points $x$ between the $\alpha$'s and the $\beta$'s which defines a bijection between the sets $\underline{\alpha}$ and $\underline{\beta}$.
It is always possible to find such a pointed Heegaard splitting after a finite number of isotopies of the $\alpha$'s and $\beta$'s curves (see \cite{Gompf1}).
The Morse function $f$ corresponding to this Heegaard splitting has one critical point of index 0 and 3 and $g$ critical points of index 1 and 2.
Denote the set of index 1 (resp. 2) critical points by $\{q_i\}_{i=1}^{g}$ (resp. $\{p_i\}_{i=1}^{g}$) where $q_i$ (resp. $p_i$) corresponds to $\alpha_i$ (resp. $\beta_i$) for all $i\in \{1,\ldots,g\}$.
The $n$-uplet of intersection points $x=(x_{1,i_1},\ldots,x_{n,i_n})$ gives for each $j\in \{1,\ldots,g\}$ an integral curve connecting $q_j$ to $p_{i_j}$.
The point $z$ gives an integral curve connecting the index 0 critical point to the index 3 critical point.

Take a combinatorial realization $(T,V)$ as given by Theorem \ref{theo:combinatorial-realization} of $(M,f)$.
Then to each point $x_{i,j_i}$ correspond now a $V$-path $\gamma$ from an hyperface of the critical cell $\sigma_{p_{i_j}}$ to $\sigma_{q_i}$: we change the matching along this path so that both $\tau_{p_{i_j}}$ and $\sigma_{q_i}$ are no more critical cells.
If $\gamma: \sigma_0,\ldots,\sigma_r = \sigma_{q_i}$, then do the following:
\begin{itemize}
 \item[$\bullet$] match $\sigma_0$ with $\tau$,
 \item[$\bullet$] for every $i\in \{1,\ldots,r\}$ match $\sigma_i$ with $V(\sigma_{i-1})$.
\end{itemize}
Now suuppose that $z$ belongs to the interior of a 2-cell $\tau_z$ (if not, subdivide $T$).
Denote by $\varsigma$ the critical cell of dimension 3 and by $\upsilon$ the critical cell of dimension 0.
There is by construction a $V$-path $\gamma : \tau_0 ,\ldots,\tau_r = \tau_z$ from an hyperface of $\varsigma$ to $\tau_z$ since $z$ is in the stable manifold of the index 3 critical point.
We modify the matching along $\gamma$ this way:
\begin{itemize}
 \item[$\bullet$] match $\tau_0$ with $\varsigma$,
 \item[$\bullet$] for every $i\in \{1,\ldots,r\}$ match $\tau_i$ with $V(\tau_{i-1})$.
\end{itemize}
In fact, it is no more a matching since $\tau_z$ belongs to two edges of the matching.

Nevertheless, $\tau_z$ belongs to the unstable manifold of $\upsilon$ the critical cell of dimension $0$.
By the construction done in Theorem \ref{theo:combinatorial-realization}, the tubular neighbourhood of the critical point of index $0$ is equal to $D^3 = \partial D^3 * \{0\}$.
The triangulation of this tubular neighbourhood is given by making the cone over $\{0\}= \upsilon$ of a triangulation of $\partial D^3$.
We modify the matching as follows.
Let $\nu$ denotes the critical $0$-cell and suppose $\overline{\tau}_z *\nu$ is the tetrahedron $ABCD$ where $A$ corresponds to $\nu$.
The collapsing $\partial D^3 * \{0\} \searrow \{0\}$ gives in particular the following matching on $ABCD$: $(BCD,ABCD)$, $(BC,ABC)$ and $(B,AB)$ ($A$ is critical).
Modify the matching by $(A,AB)$, $(B,BC)$ and $(ABC,ABCD)$.
Then, $BCD$ (which is $\tau_z$) is critical.
This gives a complete matching over $T$.
\end{proof}

\subsection{Euler structures and homologous vector fields}
Throughout this subsection we use conventions of Turaev \cite{Turaev3}.

\subsubsection{Combinatorial Euler structures}
Complete matchings have an interpretation as Euler chains.
First, we recall Euler structures as defined by Turaev \cite{Turaev3}.
Let $(M,\partial M)$ be a smooth manifold of dimension $n$ and $T$ be a $C^1$-triangulation of $M$.

Suppose $\partial M = \partial_0 M \coprod \partial_1 M$ be such that $\chi(M,\partial_0 M)=0$ and let $T_i$ be equal to $T_{\vert \partial_i M}$ for $i\in \{0,1\}$.

Denote $K$ the set of cells of $T$ and $K_i$ the set of cells of $T_i$ for $i\in \{0,1\}$.
For each cell $\sigma\in K$, let $sgn(\sigma)$ be equal to $(-1)^{dim(\sigma)}$ and pick $a_{\sigma}$ a point in the interior of $\sigma$.
An Euler chain in $(T,T_0)$ is a one-dimensional singular chain $\xi$ in $T$ with the boundary of the form $\sum_{\sigma\in K-K_0} sgn(\sigma)a_{\sigma}$.
Since $\chi(M,\partial_0 M)=0$, the set of Euler chains is non-empty.
Given two Euler chains $\xi$ and $\eta$, the difference $\xi-\eta$ is a cycle.
If   $\xi-\eta = 0 \in H_1(M)$ then we say that $\xi$ and $\eta$ are homologous.
A class of homologous Euler chains in $(T,T_0)$ is called a combinatorial Euler structure on $(T,T_0)$.
Let $Eul(T,T_0)$ be the set of Euler structures on $(T,T_0)$.
If $\xi$ is an Euler chain, denote by $[\xi]$ its class as a combinatorial Euler structure.
Euler chains behave well with respect to the subdivision of a triangulation: this allows us to consider the set $Eul(M,\partial_0 M)$ of Euler structures on $(M,\partial_0 M)$.
Taking $\xi$ an element of this set means choosing a triangulation $(T,T_0)$ of $(M,\partial_0 M)$ and considering an Euler chain on $(T,T_0)$.

\begin{remark}
 Let $\mathcal C$ be a complete matching on a $C^1$-triangulation $(T,T_0)$ of $(M,\partial_0 M)$.
Then it defines an Euler chain $[\xi_c]\in Eul(T,T_0)$: orient every edge of the matching from odd dimensional cells to even dimensional cells. Complete matchings are special Euler chains that do not pass through a cell more than one time.
\end{remark}

\subsubsection{Homologous vector fields}
By a vector field on $(M,\partial_0 M)$ we mean (except in clearly mentioned case) a non-singular continuous vector field of tangent vectors on $M$ directed into $M$ on $\partial_0 M$ and directed outwards on $\partial_1 M$.
Since $\chi(M,\partial_0 M)=0$, there exists such vector fields on $(M,\partial_0 M)$.

Vector fields $u$ and $v$ on $(M,\partial_0 M)$ are called homologous if for some closed ball $B\subset Int(M)$ the restriction of the fields $u$ and $v$ are homotopic in the class of non-singular vector fields on $M-Int(B)$ directed into $M$ on $\partial_0 M$, outwards on $\partial_1 M$, and arbitrarily on $\partial B$.
Denote by $vect(M,\partial_0 M)$ the set of homologous vector fields on $(M,\partial_0 M)$ and the class of a vector field $u$ is denoted by $[u]$.

\subsubsection{The canonical bijection}
Turaev proved the following:
\begin{theorem}[Turaev \cite{Turaev3}]
 Let $(M,\partial_0 M)$ be a smooth pair such that $dim(M)\geq 2$.
For each $C^1$-triangulation $(T,T_0)$ of the pair $(M,\partial_0 M)$ there exists a bijection
$$
\rho: Eul(T,T_0) \rightarrow vect(M,\partial_0 M)
$$
\end{theorem}

\begin{remark}
 In fact, this bijection is an $H_1(T)$-isomorphism, but we'll make no use of it.
\end{remark}

Let us describe the construction of Turaev in the case $\partial M = \emptyset$.
Let $T$ be a $C^1$-triangulation of $M$ and $T'$ be the first barycentric subdivision of $T$.
We recall the definition of the vector field $F_1$ with singularities on $M$.
For a simplex $a$ of the triangulation $T$, let $\underline{a}$ denotes its barycenter.
If $A = <\underline{a}_0,\underline{a}_1,\ldots,\underline{a}_p>$ is a simplex of the triangulation $T'$, where $a_0 < a_1 < \ldots <a_p$ are simplexes of $T$, then, at a point $x\in Int(A)$,
$$
F_1(x) = \sum_{ 0\leq i<j\leq p } \lambda_i(x) \lambda_j(x) (\underline{a}_j - x)
$$
Here $\lambda_0,\lambda_1,\ldots, \lambda_p$ are barycentric coordinates in $A$ and $\underline{a}_j - x$ is the image of the tangent vector $\underline{a}_j - x \in T_x A$ via the homeomorphism between $T$ and $M$.

Every barycenter of each cell of $T$ is a singular point of $F_1$.
Let $\mathcal M$ be a matching on the Hasse diagram of $T$: in particular, every edge of the matching connects two singular points.
Turaev proved that the index of $F_1$ in a neighbourhood of every edge of the Hasse diagram (thought as embedded in $M$ in the obvious way) is equal to zero.
Thus, if we think about combinatorial (not necessarily Morse) vector field on $T$ as a Matching on its Hasse diagram, it encodes a desingularization of the vector field $F_1$ (where critical points of $F_1$ remain if the corresponding cell is critical).
This can be done in the case of Euler chain.
More precisely, if $\xi$ is an Euler chain, let $F_{\xi}$ denote an extension of $F_1$ to a non-singular vector field.
Turaev proved that the homotopy $[F_{\xi}]\in Vect(M)$ depends only on $[\xi]\in Eul(T)$.
The map $\rho$ is defined by $\rho([\xi])=[F_\xi]$.

In the general case where $\xi$ corresponds to a matching (instead of a complete matching), we still denote by $\rho$ the map which assigns to $\xi$ the vector field $F_{\xi}$ given by the construction above.

We refer to \cite{Turaev3} in the case $\partial M \neq \emptyset$.

\subsection{$\mathcal M$-realization of Euler structures}
\begin{definition}
 Let $[\xi]\in Eul(M,\partial_0 M)$.
We say that $[\xi]$ has a $\mathcal M$-realization if there exists a $C^1$-triangulation $T$ of $(M,\partial M)$ and a matching $\eta$ on the corresponding Hasse diagram such that $[\eta] = [\xi]$.
\end{definition}

%The question arising is wether any Euler structure has a $\mathcal M$-realization.
%We prove the following:

\begin{theorem}\label{theo:m-realization}
 Any Euler structure on a smooth oriented closed riemannian 3-manifold has an $\mathcal M$-realization.
\end{theorem}

This theorem is a step toward Heegaard--Floer homology \cite{OS2,OS1}.
Recall the first steps of the construction of Heegaard--Floer homology for closed, oriented 3-manifolds.
Choose a pointed Heegaard splitting $M = U_0 \cup_{\Sigma} U_1$ (that is choose a Morse function) and fix a {${\rm Spin}\sp c$}-structure on $M$ given by an $n$-uplets of intersection points between the $\alpha$'s and the $\beta$'s.
Turaev proved that {${\rm Spin}\sp c$}-structures are in bijection with Euler structures (and so the set $vect(M)$).
Thus, Theorem \ref{theo:m-realization} together with Theorem \ref{theo:combinatorial-realization} give for a closed oriented 3-manifold a combinatorial realization of both the {${\rm Spin}\sp c$}-structure and the pointed Heegaard splitting.
The hard part remaining is to understand how holomorphic disks can be combinatorially realized.

The following two lemmas will be useful to prove Theorem \ref{theo:m-realization}.

\begin{lemma}\label{lemma:m-realization-product-cobordism}
 Let $(M\times[0,1];M\times\{0\},M\times\{1\})$ be a smooth $n+1$ product cobordism such that $T_i$ is a $C^1$-triangulation of $M\times\{i\}$ for $i\in \{0,1\}$.
Let $(T;T_0,T_1)$ be any triangulation given by lemma \ref{lemma:gluing-cobordisms} and $[\xi]\in Eul(M\times[0,1],M\times\{0\})$ the Euler structure given by the combinatorial Morse vector field induced by $T\searrow T_0$.
Then $\rho([\xi])\in vect(M\times[0,1],M\times\{0\})$ is homologous to the vector field $v:M\times[0,1] \rightarrow T(M\times[0,1])$ defined by $v(x,t) = ((x,t),dt)$.
\end{lemma}

\begin{proof}
The construction of the collapsing $T\searrow T_0$ defines a combinatorial Morse vector field pointing downwards.
Let $\xi\in Eul(T,T_0)$ be the corresponding Euler chain.
The map $\rho:Eul(M\times[0,1],M\times\{0\}) \rightarrow vect(M\times[0,1],M\times\{0\})$ sends $\xi$ to a non-singular vector field on $(M\times[0,1],M\times\{0\})$ which is, by definition of $\rho$, by the construction of the triangulation and by the definition of the combinatorial Morse vector field,  homologous to the desired vector field.
\end{proof}

\begin{lemma}\label{lemma:m-realization-cobordism-one-critical-point}
Let $f$ be a generic Morse function on a cobordism $(M;M_0,M_1)$ with exactly one critical point $p$ of index $k$ and $(T;T_0,T_1)$ be a $C^1$-triangulation of the cobordism $(T;T_0,T_1)$ given by Theorem \ref{theo:cobordism-one-critical-point}.
Let $\xi$ be the one singular chain corresponding to the matching given by the combinatorial Morse vector field on $T$.
Then $\rho(\xi)$ is homologous to the chosen Morse--Smale gradient of $f$ outside a small ball neighbourhood of the critical point $p$.
Moreover, the index of $\rho(\xi)$ at $p$ is equal to $k$.
\end{lemma}

\begin{proof}
We follow notations of Theorem \ref{theo:cobordism-one-critical-point}.
The two collapsings $T\searrow T^s_p \cup T_0$  and $T^s_p - \sigma_p \searrow (T^s_p \cap T_0)$ define a combinatorial Morse vector field with only one critical cell.
Let $\xi$ be the corresponding one singular chain.
Thus, the barycenter of this cell (which is $p$) must be a critical point of $\rho(\xi)$.
It remains to check that outside a small neighbourhood of $p$, $\rho(\xi)$ is homologous to the Morse--Smale gradient $v$.
Since outside the tubular neighbourhood $C$ of the stable manifold the triangulation is constructed by pushing it along gradient lines of $v$, we can use lemma \ref{lemma:m-realization-product-cobordism} to see that $\rho(\xi)$ is homologous to $v$ outside the tubular neighbourhood $C$ of the stable manifold.
In a small ball neighbourhood (which is a Morse chart at $p$) of the critical point $p$, the vector field $F_1$ coincides with the Morse--Smale gradient $v$.
Since $T^C \searrow \sigma_p \cup (T_0 \cap T^C)$, $\rho(\xi)$ is homologous to $v$ outside the Morse chart of $p$ .
The fact that the index of $\rho(\xi)$ at $p$ is equal to $k$ is a consequence of the definition of $F_1$.
\end{proof}

\begin{proof}[Proof of Theorem \ref{theo:m-realization}]
We apply Theorem \ref{theo:combinatorial-realization} to obtain a $C^1$-triangulation of the 3-manifold and a combinatorial Morse vector field which realizes combinatorially the Thom-Smale complex.
Then, the construction done in corollary \ref{cor:existence-matching} defines a matching which in turns defines an Euler chain $\xi$.
The map $\rho$ sends $\xi$ to a non-singular vector field which is by construction homologous to the Morse--Smale gradient of $f$.
Thus, to prove the theorem for any $[v]\in vect(M)$, it remains to find a pointed Heegaard splitting $(\Sigma_g ; \underline{\alpha} = (\alpha_1,\ldots,\alpha_g), \underline{\beta} = (\beta_1,\ldots,\beta_g);z)$ of the 3-manifold $M$ such that an $n$-uplet of intersection points $x$ corresponds to a given $[v]\in vect(M)$.
Finally, \cite[Lemma 5.2]{OS1} tells that any $[v]\in vect(M)$ can be realized in such a way.
This concludes the proof.
\end{proof}

\bibliography{article}
\bibliographystyle{amsplain}
\end{document}